\documentclass[12pt,reqno]{amsart}

\usepackage{pdfsync}

\usepackage{amssymb}

\usepackage{mathrsfs}

\usepackage{bbm}

\DeclareMathAlphabet{\mathpzc}{OT1}{pzc}{m}{it}

\usepackage[all]{xy}

\entrymodifiers={+!!<0pt,\fontdimen22\textfont2>}

\usepackage{tikz}
\usetikzlibrary{arrows,decorations.pathmorphing,decorations.pathreplacing,positioning,shapes.geometric,shapes.misc,decorations.markings,decorations.fractals,calc,patterns}

\usepackage{graphicx}

\let\oldmarginpar\marginpar
\renewcommand\marginpar[1]{\-\oldmarginpar[\raggedleft\footnotesize #1]%
{\raggedright\footnotesize #1}}

\def\cX{\mathscr{X}}
\def\cY{\mathscr{Y}}

\def\Bk{\mathbbm{k}}

\def\sA{\mathsf{A}}
\def\sB{\mathsf{B}}
\def\sC{\mathsf{C}}

\def\sE{\mathsf{E}}
\def\sF{\mathsf{F}}
\def\sG{\mathsf{G}}
\def\sH{\mathsf{H}}

\def\sM{\mathsf{M}}

\def\sP{\mathsf{P}}

\def\sS{\mathsf{S}}
\def\sT{\mathsf{T}}
\def\sU{\mathsf{U}}

\def\sX{\mathsf{X}}
\def\sY{\mathsf{Y}}

\def\add{\mathsf{add}}

\def\adots{\mathinner{\mkern1mu\raise1.0pt\vbox{\kern7.0pt\hbox{.}}\mkern2mu\raise4.0pt\hbox{.}\mkern2mu\raise7.0pt\hbox{.}\mkern1mu}}

\def\ann{\operatorname{ann}}
\def\ast{{\textstyle *}}

\def\Coker{\operatorname{Coker}}

\def\End{\operatorname{End}}

\def\Ext{\operatorname{Ext}}

\def\Hom{\operatorname{Hom}}

\def\Image{\operatorname{Im}}

\def\Ker{\operatorname{Ker}}

\def\mod{\mathsf{mod}}
\def\Mod{\mathsf{Mod}}
\def\Fac{\mathsf{Fac}}

\def\prj{\mathsf{prj}}

\def\rad{\operatorname{rad}}

\def\thick{\mathsf{thick}}

\newtheorem{Lemma}{Lemma}[section]
\newtheorem{Theorem}[Lemma]{Theorem}
\newtheorem{Proposition}[Lemma]{Proposition}
\newtheorem{Corollary}[Lemma]{Corollary}

\theoremstyle{definition}
\newtheorem{Definition}[Lemma]{Definition}

\newtheorem{Remark}[Lemma]{Remark}

\begin{document}

\setlength{\parindent}{0pt}
\setlength{\parskip}{7pt}

\title[Intermediate co-t-structures]{Intermediate co-t-structures, two-term silting objects, $\tau$-tilting modules, and torsion classes}

\author{Osamu Iyama}
\address{Graduate School of Mathematics, Nagoya University Chikusa-ku,
  Nagoya, 464-8602 Japan} 
\email{iyama@math.nagoya-u.ac.jp}
\urladdr{http://www.math.nagoya-u.ac.jp/~iyama}

\author{Peter J\o rgensen}
\address{School of Mathematics and Statistics, Newcastle University,
  Newcastle NE1 7RU}
\email{peter.jorgensen@ncl.ac.uk}
\urladdr{http://www.staff.ncl.ac.uk/peter.jorgensen}

\author{Dong Yang}
\address{Department of Mathematics, Nanjing University, Nanjing 210093, P. R. China}
\email{dongyang2002@googlemail.com}



\thanks{2010 {\em Mathematics Subject Classification.} 18E30, 18E40}

\begin{abstract} 

  If $( \sA , \sB )$ and $( \sA^{\prime } , \sB^{ \prime } )$ are
  co-t-structures of a triangulated category, then $( \sA^{\prime } ,
  \sB^{ \prime } )$ is called intermediate if $\sA \subseteq \sA^{
    \prime }\subseteq \Sigma \sA$.  Our main results show that
  intermediate co-t-structures are in bijection with
  two-term silting subcategories, and also with support $\tau$-tilting
  subcategories under some assumptions.
  We also show that support $\tau$-tilting subcategories are in bijection
  with certain finitely generated torsion classes.
  These results generalise work by Adachi, Iyama, and Reiten.

\end{abstract}

\maketitle

\setcounter{section}{-1}
\section{Introduction}
\label{sec:introduction}

The aim of this paper is to discuss the relationship between
the following objects.
\begin{itemize}

  \item Intermediate co-t-structures.

\smallskip

  \item Two-term silting subcategories.

\smallskip

  \item Support $\tau$-tilting subcategories.

\smallskip

  \item Torsion classes.
\end{itemize}
The motivation is that if $\sT$ is a triangulated category with
suspension functor $\Sigma$ and $( \sX , \sY )$ is a t-structure of
$\sT$ with heart $\sH = \sX \cap \Sigma \sY$, then there is a
bijection between ``intermediate'' t-structures $( \sX^{ \prime } ,
\sY^{ \prime } )$ with $\Sigma \sX \subseteq \sX^{ \prime } \subseteq
\sX$ and torsion pairs of $\sH$.  This is due to \cite[thm.\ 3.1]{BR}
and \cite[prop.\ 2.1]{HRS}; see \cite[prop.\ 2.3]{W}.

We will study a co-t-structure analogue of this which also involves
silting subcategories, that is, full subcategories $\sS \subseteq \sT$ with
thick closure equal to $\sT$ which satisfy $\Hom_{ \sT }( \sS,\Sigma^i
\sS ) = 0$ for $i \geq 1$.  Silting subcategories are a useful
generalisation of tilting subcategories.

The following result follows from the bijection between bounded
co-t-structures and silting subcategories in \cite[cor.\ 5.9]{MSSS}.
See \cite{P} and \cite{AI} for background on co-t-structures and
silting subcategories.  Note that the \emph{co-heart} of a
co-t-structure $( \sA , \sB )$ is $\sA \cap \Sigma^{-1}\sB$.  If
$\sF$, $\sG$ are full subcategories of a triangulated category then
$\sF * \sG$ denotes the full subcategory of objects $e$ which permit a
distinguished triangle $f \rightarrow e \rightarrow g$ with $f \in
\sF$, $g \in \sG$.
\begin{Theorem}[=Theorem \ref{t:co-t-str-and-silting}]
Let $\sT$ be a triangulated category, $( \sA , \sB )$ a bounded
co-t-structure of $\sT$ with co-heart $\sS$. Then we have a
bijection between the following sets.
\begin{itemize}

  \item[(i)] Co-t-structures $( \sA^{\prime } , \sB^{
    \prime } )$ of $\sT$ with $\sA \subseteq \sA^{ \prime }\subseteq
  \Sigma \sA$.

\smallskip

  \item[(ii)] Silting subcategories of $\sT$ which are in $\sS *
    \Sigma \sS$. 

\end{itemize}
\end{Theorem}
The co-t-structures in (i) are called \emph{intermediate}.  The silting
subcategories in (ii) are called \emph{two-term}, motivated by the existence
of a distinguished triangle $s_1 \rightarrow s_0 \rightarrow s'$ with
$s_i \in \sS$ for each $s' \in \sS'$.  The theorem reduces the study
of intermediate co-t-structures to the study of two-term silting
subcategories.

Our main results on two-term silting subcategories and $\tau$-tilting
theory can be summed up as follows. 
We extend the notion of support $\tau$-tilting modules
for finite dimensional algebras over fields given in \cite{AIR} to that
for essentially small additive categories, see Definition
\ref{definition:support-tau-tilting} and \ref{def:support_tau-tilting}.
For a commutative ring $\Bk$, we say that a $\Bk$-linear category
is \emph{Hom-finite} if each Hom-set is a finitely generated $\Bk$-module.

\begin{Theorem}
[=Theorems \ref{silting and support tau tilting} and \ref{thm:main}]
\label{main 2}
Let $\sT$ be a triangulated category with a silting subcategory $\sS$.
Assume that each object of $\sS * \Sigma\sS$ can be written as a
direct sum of indecomposable objects unique up to isomorphism.  Then
there is a bijection between the following sets.
\begin{itemize}

  \item[(i)]  Silting subcategories of $\sT$ which are in $\sS *
    \Sigma \sS$. 

\smallskip

  \item[(ii)] Support $\tau$-tilting pairs of $\mod\, \sS$.

\end{itemize}

If $\sT$ is Krull--Schmidt, $\Bk$-linear and Hom-finite over
a commutative ring $\Bk$,
and $\sS =\add\, s$ for a silting object $s$, then there is a bijection between
the following sets. 

\begin{itemize}

  \item[(iii)]  Basic silting objects of $\sT$ which are in $\sS *
  \Sigma \sS$, modulo isomorphism.

\smallskip

  \item[(iv)]  Basic support $\tau$-tilting modules of $\mod\, E$, modulo
  isomorphism, where $E=\End_{\sT}(s)$.
\end{itemize}
Note that in this case, there is a bijection between {\rm (i)} and
{\rm (iii)} by \cite[prop.\ 2.20 and lem.\ 2.22(a)]{AI}.
\end{Theorem}

Note that Theorem \ref{main 2} is a much stronger version
of \cite[thm.\ 3.2]{AIR}, where $\sT$ is assumed to be the homotopy
category of bounded complexes of finitely generated projective
modules over a finite dimensional algebra $\Lambda$ over a field and
$s$ is assumed to be $\Lambda$.

Moreover, we give the following link between $\tau$-tilting theory and
torsion classes. Our main result shows that
support $\tau$-tilting pairs correspond
bijectively with certain finitely generated torsion classes,
which is a stronger version of \cite[thm.\ 2.7]{AIR}.
Note that $\Fac\,\sM$ is the subcategory of
$\Mod\,\sC$ consisting of factor objects of finite direct sums of
objects of $\sM$, and $\sP( \sT )$ denotes the $\Ext$-projective
objects of $\sT$, see Definition \ref{def:tors}.

\begin{Theorem}
[=Theorem \ref{tau tilting and torsion class 2}]
\label{tau tilting and torsion class}
Let $\Bk$ be a commutative noetherian local ring,
$\sC$ an essentially small, Krull-Schmidt, $\Bk$-linear Hom-finite category.  There is  
a bijection $\sM \mapsto \Fac\,\sM$ from the first to the second of
the following sets.
\begin{itemize}
  \item[(i)] Support $\tau$-tilting pairs $(\sM,\sE)$ of $\mod\,\sC$.

\smallskip

  \item[(ii)] Finitely generated torsion classes $\sT$ of $\Mod\,\sC$
    such that each finitely generated projective $\sC$-module has a
    left $\sP(\sT)$-approximation. 

\end{itemize}
\end{Theorem}

\section{Basic definitions}

Let $\sC$ be an additive category.  When we say that $\sU$ is a
\emph{subcategory} of $\sC$, we always assume $\sU$ is full and closed
under finite direct sums and direct summands.  For a collection $\sU$ of
objects of $\sC$, we denote by $\add\, \sU$ the smallest subcategory
of $\sC$ containing $\sU$.

Let $\sC$ be an essentially small additive category.
We write $\Mod\, \sC$ for the abelian category of contravariant
additive functors from $\sC$ to the category of abelian groups and
$\mod\, \sC$ for the full subcategory of finitely presented functors,
see \cite[pp.\ 184 and 204]{A}.

The suspension functor of a triangulated category is denoted by
$\Sigma$. 

We first recall the notions of co-t-structures and silting subcategories.

\begin{Definition}\label{d:co-t-str}
Let $\sT$ be a triangulated category.
A \emph{co-$t$-structure} on
$\sT$  is a pair $(\sA,\sB)$ of full  subcategories of $\sT$ such that
\begin{itemize}
\item[(i)] $\Sigma^{-1}\sA\subseteq\sA$ and
$\Sigma\sB\subseteq\sB$;

\smallskip

\item[(ii)] $\Hom_{\sT}(a,b)=0$ for $a\in\sA$
and $b\in\sB$,

\smallskip

\item[(iii)] for each $t\in\sT$ there is a triangle $a\rightarrow
t\rightarrow b\rightarrow\Sigma a$ in $\sT$ with $a\in\sA$ and $b\in\sB$.
\end{itemize}
The \emph{co-heart} is defined as the intersection $\sA~\cap~\Sigma^{-1}\sB$. See \cite{P,Bondarko10}.
\end{Definition}

\begin{Definition}
Let $\sT$ be a triangulated category.
\begin{itemize}
\item[(i)] A subcategory $\sU$ of $\sT$ is called a
\emph{presilting subcategory} if $\sT(u,\Sigma^{\geq 1}u')=0$ holds
for any $u,u'\in\sU$.
\smallskip

\item[(ii)] A presilting subcategory $\sS \subseteq \sT$ is a \emph{silting
  subcategory} if $\thick( \sS ) = \sT$, see \cite[def.\ 2.1(a)]{AI}.
  Here $\thick( \sS )$ denotes the
  smallest thick subcategory of $\sT$ containing $\sS$.
\smallskip

\item[(iii)] An object $u \in \sT$ is called a \emph{presilting object}
if it satisfies $\sT( u,\Sigma^{ \geq 1 }u ) = 0$, namely, if $\add(u)$
is a presilting subcategory. Similarly an object $u \in \sT$ is called
a \emph{silting object} if $\add(u)$ is a silting subcategory.
\end{itemize}
\end{Definition}

Next we introduce the notion of support $\tau$-tilting subcategories.

\begin{Definition}
\label{definition:support-tau-tilting}
Let $\sC$ be an essentially small additive category.
\begin{itemize}

\item[(i)] Let $\sM$ be a subcategory of $\mod\,\sC$.  A class $\{\,
  P_1 \stackrel{ \pi^m }{ \rightarrow } P_0 \rightarrow m \rightarrow
  0 \,\mid\, m \in \sM \,\}$ of projective presentations in $\mod\,
  \sC$ is said to have \emph{Property (S)} if
\[
  \Hom_{ \mod\,\sC }( \pi^m , m' )
  : \Hom_{ \mod\,\sC }( P_0 , m' ) 
    \rightarrow \Hom_{ \mod\, \sC }( P_1 , m' )
\]
is surjective for any $m , m' \in \sM$.

\smallskip

\item[(ii)] A subcategory $\sM$ of $\mod\, \sC$ is said to be
  \emph{$\tau$-rigid} if there is a class of projective presentations
  $\{P_1 \rightarrow P_0 \rightarrow m \rightarrow 0\mid m\in\sM\}$
  which has Property (S).

\smallskip

\item[(iii)] A \emph{$\tau$-rigid pair} of $\mod\, \sC$ is a pair $(
  \sM , \sE )$, where $\sM$ is a $\tau$-rigid subcategory of $\mod\,
  \sC$ and $\sE \subseteq \sC$ is a subcategory with $\sM( \sE )=0$,
  that is, $m( e )=0$ for each $m \in \sM$ and $e \in \sE$.

\smallskip

\item[(iv)] A $\tau$-rigid pair $( \sM , \sE )$ is \emph{support
  $\tau$-tilting} if $\sE = \Ker( \sM )$ and for each $s \in \sC$
  there exists an exact sequence $\sC( - , s )
  \stackrel{f}{\rightarrow} m^0 \rightarrow m^1 \rightarrow 0$ with
  $m^0, m^1 \in \sM$ such that $f$ is a left $\sM$-approximation.

\end{itemize}
\end{Definition}

It is useful to recall the notion of Krull--Schmidt categories.

\begin{Definition}
\label{def:hash}
An additive category $\sC$ is called \emph{Krull--Schmidt} if each
of its objects is the direct sum of finitely many objects with local
endomorphism rings.  It follows that these finitely many objects are
indecomposable and determined up to isomorphism, see \cite[thm.\
I.3.6]{B}.  It also follows that $\sC$ is \emph{idempotent complete},
that is, for an object $c$ of $\sC$ and an idempotent $e\in\sC(c,c)$,
there exist objects $c_1$ and $c_2$ such that $c=c_1\oplus c_2$ and
$e=\mathrm{id}_{c_1}$, see \cite[5.1]{K}.
\begin{itemize}

\item[(i)] For $c \in \sC$ to be \emph{basic} means that it has no
  repeated indecomposable direct summands.

\smallskip

\item[(ii)] For an object $c \in \sC$, let $\#_{ \sC }( c )$ denote
  the number of pairwise non-isomorphic indecomposable direct summands of $c$.

\end{itemize}
\end{Definition}

The following is a ring version of Definition \ref{definition:support-tau-tilting}.

\begin{Definition}
\label{def:support_tau-tilting}
Let $E$ be a ring such that $\mod\, E$ is Krull--Schmidt.
\begin{itemize}

\item[(i)] A module $U \in \mod\, E$ is called \emph{$\tau$-rigid} if
  there is a projective presentation $P_1
  \stackrel{\pi}{\rightarrow} P_0 \rightarrow U \rightarrow 0$ in
  $\mod\, E$ such that $\Hom_E(\pi,U)$ is surjective.

\smallskip

\item[(ii)] A $\tau$-rigid module $U \in \mod\, E$
  is called \emph{support $\tau$-tilting} if
  there is an idempotent $e \in E$ which satisfies that $Ue
  = 0$ and $\#_{ \mod\, E }( U ) = \#_{ \prj\, (E / EeE) }( E / EeE )$.

\end{itemize}
\end{Definition}

\begin{Remark}
Part (ii) of the definition makes sense because $\prj\, (E / EeE)$ is
Krull--Schmidt.  Namely, since $\mod\, E$ is Krull--Schmidt, it follows
that $\prj\, E$ is Krull--Schmidt with additive generator $E_E$.  The
same is hence true for $(\prj\, E)/[\add\, eE]$ for each idempotent $e
\in E$, and it is not hard to check that the endomorphism ring of
$E_E$ in $(\prj\, E)/[\add\, eE]$ is $E / EeE$, so there is an
equivalence of categories
\[
  (\prj\, E)/[\add\, eE] \stackrel{ \sim }{ \rightarrow }
  \prj\, (E / EeE).
\]
Hence $\prj\, (E / EeE)$ is Krull--Schmidt.

If $E$ is a finite-dimensional algebra over a field, then the
definition coincides with the original definition of basic support
$\tau$-tilting modules by Adachi, Iyama and Reiten \cite[def.\
0.1(c)]{AIR}.
\end{Remark}

Finally we introduce the notion of torsion classes.

\begin{Definition}
\label{def:tors}
Let $\sC$ be an essentially small additive category, $\sT$ a full subcategory of
$\Mod\, \sC$. 
\begin{itemize}

\item[(i)] We say that $\sT$ is a \emph{torsion class} if it is closed
  under factor modules and extensions.

\smallskip

\item[(ii)] For a subcategory $\sM$ of $\Mod\,\sC$, we denote by
  $\Fac\,\sM$ the subcategory of $\Mod\,\sC$ consisting of factor
  objects of objects of $\sM$.

\smallskip

\item[(iii)] We say that a torsion class $\sT$ is \emph{finitely generated}
  if there exists a full subcategory $\sM$ of $\mod\,\sC$
  such that $\sT=\Fac\,\sM$. Clearly the objects in
  $\Fac\,\sM$ are  finitely generated $\sC$-modules, which are not necessarily finitely
  presented.

\smallskip

\item[(iv)] An object $t$ of a torsion class $\sT$ is
  \emph{Ext-projective} if $\Ext^1_{\Mod\,\sC}(t,\sT)=0$.  We denote
  by $\sP(\sT)$ the full subcategory of $\sT$ consisting of all
  Ext-projective objects of $\sT$.

\end{itemize}
\end{Definition}

\section{Silting subcategories and co-t-structures}
In this section, $\sT$ is an essentially small, idempotent complete
triangulated category.

Let $(\sA,\sB)$ be a co-t-structure on $\sT$. It follows from the
definition that  
\begin{eqnarray*}
\sA&=&\{t\in\sT~|~\Hom(t,b)=0~\forall b\in\sB\},\\
\sB&=&\{t\in\sT~|~\Hom(a,t)=0~\forall a\in\sA\}.
\end{eqnarray*} 
In particular, both $\sA$ and $\sB$ are idempotent complete and extension closed. Hence so is the co-heart $
\sS=\sA\cap\Sigma^{-1}\sB$. Set
\[
\sS*\Sigma\sS=\{t\in\sT~|~\text{there is a triangle } s_1\rightarrow
s_0\rightarrow t\rightarrow \Sigma s_1~\text{with } s_0, s_1 \in \sS \}\subseteq \sT.
\]

The following lemma will often be used without further remark.
\begin{Lemma}\label{l:2-co-heart} There is an equality
$\sS*\Sigma\sS=\Sigma\sA\cap\Sigma^{-1}\sB$. As a consequence, $\sS*\Sigma\sS$ is
idempotent complete and extension closed.
\end{Lemma}
\begin{proof} The inclusion
$\sS*\Sigma\sS\subseteq\Sigma\sA\cap\Sigma^{-1}\sB$ is clear, because both $\sS$ and $\Sigma\sS$ are contained in $\Sigma\sA\cap\Sigma^{-1}\sB$, which is extension closed. Next we show the opposite inclusion. Let $t\in\Sigma\sA\cap\Sigma^{-1}\sB$. Then by Definition~\ref{d:co-t-str}(iii) there is a triangle $a\rightarrow t\rightarrow b\rightarrow\Sigma a$ with $a\in\sA$ and $b\in\sB$. Since both $t$ and $\Sigma a$ belong to $\Sigma\sA$, so is $b$ due to the fact that $\sA$ is extension closed. Thus $b\in\Sigma\sA\cap\sB=\Sigma\sS$. Similarly, one shows that $a\in\sS$. Thus we obtain a triangle $\Sigma^{-1}b\rightarrow a\rightarrow t\rightarrow b$ with $\Sigma^{-1}b$ and $a$ in $\sS$, meaning that $t\in\sS*\Sigma\sS$.
\end{proof}

It is easy to see that $\Hom(s,\Sigma^{\geq 1}s')=0$  for any
$s,s'\in\sS$. That is, $\sS$ is a presilting subcategory of $\sT$. The co-$t$-structure $(\sA,\sB)$ is said to be \emph{bounded} if
\[
\bigcup_{n\in\mathbb{Z}}\Sigma^n \sB=\sT=\bigcup_{n\in\mathbb{Z}}\Sigma^n\sA.
\]

\begin{Theorem}\emph{(\cite[cor. 5.9]{MSSS})} \label{t:co-t-str-and-silting}
There is a bijection $(\sA,\sB)\mapsto \sA\cap\Sigma^{-1}\sB$ from the first to the second of the following sets.
\begin{itemize}

\item[(i)] Bounded co-t-structures on $\sT$.

\smallskip

\item[(ii)] Silting subcategories of $\sT$.
\end{itemize}
\end{Theorem}

This result has the following consequence.

\begin{Theorem}\label{t:inter-co-t-str-and-2-term-silting}
Let $(\sA,\sB)$ be a bounded co-t-structure on $\sT$ with co-heart $\sS$. Then there is a bijection $(\sA',\sB')\mapsto \sA'\cap\Sigma^{-1}\sB'$ from the first to the second of the following sets.
\begin{itemize}
\item[(i)] Bounded co-t-structures $(\sA',\sB')$ on $
\sT$ with $\sA\subseteq\sA'\subseteq\Sigma \sA$.

\smallskip

\item[(ii)] Silting subcategories of $\sT$ which are in $\sS*\Sigma\sS$.
\end{itemize}
\end{Theorem}
\begin{proof}
Let  $(\sA',\sB')$ be a bounded co-t-structure on $\sT$ with $\sA\subseteq\sA'\subseteq\Sigma \sA$. Then $\sB\supseteq\sB'\supseteq\Sigma\sB$. It follows that $\sA'\cap\Sigma^{-1}\sB'\subseteq\Sigma\sA\cap\Sigma^{-1}\sB=\sS*\Sigma\sS$. The last equality is by Lemma~\ref{l:2-co-heart}.

Let $\sS'$ be a silting subcategory of $\sT$ which is in $\sS*\Sigma\sS$. Let $\sA'$ be the smallest extension closed subcategory of $\sT$ containing $
\Sigma^{\leq 0}\sS'$ and $\sB'$ be the smallest extension closed subcategory of $\sT$ containing $\Sigma^{\geq 1}\sS'$. Then $(\sA',\sB')$ is the bounded co-t-structure corresponding to $\sS'$ as in Theorem~\ref{t:co-t-str-and-silting}, see \cite[cor. 5.9]{MSSS}. Since $\sS'\subseteq \sS*\Sigma\sS$, it follows that $\sA'$ is contained in the smallest extension closed subcategory of $\sT$ containing $\Sigma^{\leq 1}\sS$, which is exactly $\Sigma\sA$. Similarly, one shows that $\sB'$ is contained in $\sB$, implying that $\sA'$ contains $\sA$. Thus, $\sA\subseteq\sA'\subseteq\Sigma\sA$.
\end{proof}

The co-t-structures in (i) are called  \emph{intermediate} with respect to $(\sA,\sB)$. The silting subcategories in (ii) are called \emph{2-term} with respect to $\sS$. Clearly, if $(\sA',\sB')$ is intermediate with respect to $(\sA,\sB)$, then $(\sA,\sB)$ is intermediate with respect to $(\Sigma^{-1}\sA',\Sigma^{-1}\sB')$. The next result is a corollary of Theorems~\ref{t:co-t-str-and-silting} and~\ref{t:inter-co-t-str-and-2-term-silting}.

\begin{Corollary}\label{c:shifting-2-term}
Let $\sS$ and $\sS'$ be two silting subcategories of $\sT$. If $\sS'$ is 2-term with respect to $\sS$, then $\sS$ is 2-term with respect to $\Sigma^{-1}\sS'$.
\end{Corollary}

\section{Two-term silting subcategories and support $\tau$-tilting pairs}

In this section, $\sT$ is an essentially small, idempotent complete
triangulated category, and $\sS \subseteq \sT$ is a silting
subcategory.

\begin{Remark}
\label{rmk:blanket-general}
\begin{itemize}

  \item[(i)] There is a functor 
\[
  F\;:\;\;\sT \rightarrow \Mod\, \sS
  \;\; , \;\;
  t \mapsto \sT( -,t )\mid_{ \sS },
\]
sometimes known as the restricted Yoneda functor.
\label{remark:restricted-yoneda-functor}

\smallskip

  \item[(ii)]  By Yoneda's Lemma, for
    $M \in \Mod\, \sS$ and $s \in \sS$, there is a
    natural isomorphism
\[
  \Hom_{ \Mod\, \sS }\big( \sS( -,s ) , M \big)
  \stackrel{ \sim }{ \rightarrow } M( s );
\]
see \cite[p.\ 185]{A}.

\smallskip

  \item[(iii)] By \cite[prop.\ 6.2(3)]{IY} the functor $F$ from (i)
    induces an equivalence 
\begin{equation}
\label{equ:mod_equivalence}
  ( \sS * \Sigma \sS ) / [ \Sigma \sS ]
  \stackrel{\sim}{\rightarrow} \mod\, \sS.
\end{equation}
To get this from \cite[prop.\ 6.2(3)]{IY}, set $\cX = \sS$, $\cY =
\Sigma \sS$, and observe that the proof works in the generality of the
present paper.

\end{itemize}
\end{Remark}


\begin{Lemma}
\label{lem:S-general}
Let $\sU$ be a full subcategory of $\sS*\Sigma\sS$.  For
$u \in \sU$ let 
\begin{equation}
\label{equ:triangle-general}
  s^u_1 \stackrel{ \sigma }{ \rightarrow } s^u_0 \rightarrow u \rightarrow \Sigma s^u_1
\end{equation}
be a distinguished triangle in $\sT$ with $s^u_0, s^u_1 \in \sS$.
Applying the functor $F$ gives a projective presentation
\begin{equation}
\label{equ:presentation-general}
  P^U_1 \stackrel{ \pi^u }{ \rightarrow } P^U_0 \rightarrow U \rightarrow 0
\end{equation}
in $\mod\, \sS$, and
\begin{align*}
 {\mbox{ $\sU$ is a presilting subcategory }} \Leftrightarrow
  \mbox{ the class $\{\, \pi^u \,|\, u \in \sU \,\}$ has Property (S). } 
\end{align*}
\end{Lemma}

\begin{proof}
Clearly $F$ applied to the distinguished
triangle \eqref{equ:triangle-general} gives the projective
presentation \eqref{equ:presentation-general}.

To get the bi-implication in the last line of the lemma, first note that for $u,u'\in\sU$ we have
\begin{equation}
\label{equ:Ext_geq_2_vanishing_general}
  \sT( u,\Sigma^{ \geq 2 }u' ) = 0
\end{equation}
since $u,u' \in \sS * \Sigma \sS$. 

By Remark \ref{rmk:blanket-general}(ii), the map
$\Hom_{\mod\,\sS}(\pi,F(u'))$ is the same as
\begin{equation}
\label{equ:S_picture_morphism_general}
  \sT( s^u_0,u' ) \rightarrow \sT( s^u_1,u' ).
\end{equation}
So the class $\{\, \pi^u \,|\, u\in \sU \,\}$ has Property (S) if and
only if the morphism \eqref{equ:S_picture_morphism_general} is
surjective for all $u , u' \in \sU$.  However, the distinguished
triangle \eqref{equ:triangle-general} induces an exact sequence
\[
  \sT( s^u_0,u' )
  \rightarrow \sT( s^u_1,u' )
  \rightarrow \sT( \Sigma^{ -1 }u,u' )
  \rightarrow \sT( \Sigma^{ -1 }s^u_0,u' )
\]
where the last module is $0$ since $u' \in \sS * \Sigma \sS$.  So
\eqref{equ:S_picture_morphism_general} is surjective if and only if $\sT(
\Sigma^{-1}u,u') \cong \sT( u,\Sigma u' ) = 0$.  This happens  for all $u,u'\in\sU$  if and
only if $\sU$ is presilting because of equation
\eqref{equ:Ext_geq_2_vanishing_general}.
\end{proof}

\begin{Theorem}
\label{thm:main-general}
The functor $F : \sT \rightarrow \Mod\, \sS$ induces a
surjection
\[
  \Phi : \sU
  \mapsto \big( F( \sU ) , \sS \cap \Sigma^{-1} \sU \big)
\]
from the first to the second of the following sets.
\begin{itemize}

\item[(i)] Presilting subcategories of $\sT$ which are contained in
  $\sS * \Sigma \sS$.

\smallskip

\item[(ii)] $\tau$-rigid pairs of $\mod\, \sS$.

\end{itemize}
It restricts to a surjection $\Psi$ from the first to the second of the following sets.
\begin{itemize}
\item[(iii)] Silting subcategories of $\sT$ which are contained in $\sS*\Sigma\sS$.
\smallskip
\item[(iv)] Support $\tau$-tilting pairs of $\mod\, \sS$.
\end{itemize}
\end{Theorem}

\begin{proof}
We need to prove
\begin{itemize}
\item[(a)] The map $\Phi$ has values in $\tau$-rigid pairs of $\mod\,
\sS$. 
\smallskip
\item[(b)] The map $\Phi$ is surjective.
\smallskip
\item[(c)] The map $\Psi$ has values in support $\tau$-tilting pairs
of $\mod\, \sS$.
\smallskip
\item[(d)] The map $\Psi$ is surjective.
\end{itemize}

(a) Let $\sU$ be a presilting subcategory of $\sT$ which is contained
in $\sS*\Sigma\sS$. For each $u\in\sU$, there is a distinguished
triangle $s_1\rightarrow s_0\rightarrow u\rightarrow \Sigma s_1 $ with
$s_0,s_1\in\sS$. Lemma \ref{lem:S-general} says that $F$ sends the set
of these triangles to a set of projective presentations
\eqref{equ:presentation-general} which has Property (S) because $\sU$
is presilting. It remains to show that for $u\in\sU$ and $u'\in
\sS\cap \Sigma^{-1} \sU$ we have $F(u)(u')=0$. This is
again true because $F(u)(u')=\sT(u',u)$ holds and $\sU$ is presilting.

\smallskip

(b) Let $(\sM,\sE)$ be a $\tau$-rigid pair of $\mod\,\sS$.  For
each $m \in \sM$ take a projective
presentation
\begin{equation}
\label{equation:presentation-in-surjectivity}
  P_1
  \stackrel{\pi^m}{\rightarrow} P_0
  \rightarrow m
  \rightarrow 0
\end{equation}
such that the class $\{\, \pi^m \,|\, m \in \sM \,\}$ has Property
(S). By Remark \ref{rmk:blanket-general}(ii) there is a unique morphism
$f_m: s_1 \rightarrow s_0$ in $\sS$ such that $F( f_m ) =
\pi^m$. Moreover, $F( \mathrm{cone}(f_m) ) \cong m$. Since
\eqref{equation:presentation-in-surjectivity} has Property (S), it
follows from Lemma \ref{lem:S-general} that the category
\[
  \sU_1 := \{\, \mathrm{cone}(f_m) \,|\, m \in \sM \,\}
\]
is a presilting subcategory, and $\sU_1 \subseteq \sS*\Sigma\sS$ is
clear.  Let $\sU$ be the additive hull of $\sU_1$ and $\Sigma\sE$ in
$\sS*\Sigma\sS$. 
Now we show that $\sU$ is a presilting subcategory of $\sT$.
Let $e\in\sE$. Clearly we have $\sT(\mathrm{cone}(f_m)\oplus\Sigma e,\Sigma^2e)=0$.
Applying $\sT(e,-)$ to a triangle
$s_1\xrightarrow{f_m}s_0\to \mathrm{cone}(f_m)\to\Sigma s_1$,
we have an exact sequence
\[\sT(e,s_1)\xrightarrow{f_m}\sT(e,s_0)\to\sT(e,\mathrm{cone}(f_m))\to0,\]
which is isomorphic to $P_1(e)\xrightarrow{\pi^m}P_0(e)\to m(e)\to0$
by Remark \ref{rmk:blanket-general}(ii).
The condition $\sM(\sE)=0$ implies that $\sT(e,\mathrm{cone}(f_m))=0$.
Thus the assertion follows.
It is clear that $\Phi( \sU ) =( \sM , \sE )$.

\smallskip

(c) Let $\sU$ be a silting subcategory of $\sT$ which is contained in
$\sS*\Sigma\sS$.

Let $s\in\sS$ be an object of $\Ker F(\sU)$, i.e. $\sT( s , u )=0$ for each
$u \in \sU$. This implies that $\sU \oplus \add( \Sigma s )$ is also a
silting subcategory of $\sT$ in $\sS*\Sigma \sS$. It follows from
\cite[Theorem 2.18]{AI} that $\Sigma s$ belongs to $\sU$ whence $s$
belongs to $\Sigma^{-1}\sU$ and hence to $\sS\cap\Sigma^{-1}\sU$. This
shows the inclusion $\Ker F(\sU) \subseteq \sS \cap \Sigma^{-1}
\sU$. The reverse inclusion was shown in (a), so $\Ker F(\sU) = \sS
\cap \Sigma^{-1} \sU$.

By Corollary~\ref{c:shifting-2-term}, we have $\sS\subseteq (\Sigma^{-1}\sU)*\sU$. 
In particular, for $s\in\sS$, there is a distinguished triangle
\begin{equation}
\label{equation:triangle-copresentation}
  s
  \rightarrow u^0
  \rightarrow u^1
  \rightarrow \Sigma s
\end{equation}
Applying $F$ we obtain an exact sequence
\begin{equation}
  F(s)
  \stackrel{f}{\rightarrow} F(u^0)
  \rightarrow F(u^1)
  \rightarrow 0.
\end{equation}
For each $u \in \sU$ there is the following commutative diagram.
\[
  \xymatrix{
    \sT(u^0,u)\ar[r]\ar[d] & \sT(s,u)\ar[r] \ar[d] & \sT(u^1,\Sigma u)=0\\
    \Hom_{\mod\,\sS}(F(u^0),F(u))\ar[r]_-{f^*} & \Hom_{\mod\,\sS}(F(s),F(u))
           }
\]
The right vertical map is induced from the Yoneda embedding, so it is
bijective. It follows that $f^*$ is surjective, that is, $f$ is a left
$F(\sU)$-approximation.  Altogether, we have shown that $\Phi( \sU )$
is a support $\tau$-tilting pair of $\mod\, \sS$. 

\smallskip

(d)  Let $(\sM,\sE)$ be a support $\tau$-tilting pair of $\mod\,\sS$ and let $\sU$ be the preimage of $(\sM,\sE)$ under $\Phi$ constructed in (b). 

By definition, for each $s\in\sS$, there is an exact sequence
$F(s)\stackrel{f}{\rightarrow} F(u_s^0)\rightarrow F(u_s^1)\rightarrow
0$ such that $u_s^0,u_s^1\in\sU$ and $f$ is a left
$F(\sU)$-approximation.  By Yoneda's lemma, there is a unique morphism
$\alpha:s\rightarrow u_s^0$ such that $F( \alpha ) = f$.  Form the
distinguished triangle
\begin{equation}
\label{equation:triangle-approxi}
  s
  \stackrel{\alpha}{\rightarrow} u_s^0
  \rightarrow t_s
  \rightarrow \Sigma s.
\end{equation}
Let $\widetilde{\sU}$ be the additive closure of $\sU$ and $\{\, t_s
\,|\, s \in \sU \,\}$. We claim that $\widetilde{\sU}$ is a silting
subcategory of $\sT$ contained in $\sS*\Sigma\sS$ such that $\Phi(
\widetilde{\sU} ) = ( \sM , \sE )$.

First, $t_s \in u_s^0*\Sigma s \subseteq \sS*\Sigma \sS$. Therefore,
$\widetilde{\sU} \subseteq \sS * \Sigma \sS$.

Secondly, by applying $F$ to the triangle
\eqref{equation:triangle-approxi}, we see that $F(t_s)$ and $F(u_s^1)$
are isomorphic in $\mod\,\sS$. For $u\in\sU$, consider the following
commutative diagram.
\[
  \xymatrix{
    \sT(u_s^0,u) \ar[r]^-{\alpha^*} \ar[d]_{F(-)} 
      & \sT(s,u)\ar[r]\ar[d]^{\cong} 
      & \sT(t_s,\Sigma u)\ar[r] 
      & \sT(u_s^0,\Sigma u)=0 \\
    \Hom_{\mod\,\sS}( F(u_s^0) , F(u) ) \ar[r]_-{f^*}
      & \Hom_{\mod\,\sS}(F(s),F(u))
           }
\]
By Remark~\ref{rmk:blanket-general}(iii), the map $F(-)$ is surjective. 
Because $f$ is a left
$F(\sU)$-approximatiom, $f^*$ is also surjective. So $\alpha^*$ is surjective too,
implying that $\sT(t_s,\Sigma u)=0$.  On the other hand, applying
$\sT(u,-)$ to the triangle \eqref{equation:triangle-approxi}, we
obtain an exact sequence
\[
\sT(u,\Sigma u_s^0)\rightarrow \sT(u,\Sigma t_s)\rightarrow \sT(u,\Sigma^2 s).
\] 
The two outer terms are trivial, hence so is the middle
term. Moreover, if $s' \in \sS$ then applying $\sT(t_{s'},-)$ to the
triangle \eqref{equation:triangle-approxi} gives an exact sequence
\[
\sT(t_{s'},\Sigma u_s^0)\rightarrow \sT(t_{s'},\Sigma t_s)\rightarrow \sT(t_{s'},\Sigma^2 s).
\]
The two outer terms are trivial, hence so is the middle term. It
follows that $\widetilde{\sU}$ is presilting. It is then silting
because it generates $\sS$.

Thirdly, $F( \widetilde{\sU} ) = F( \sU )$ because $F( t_s ) \cong F(
u_s^1 )$.

Finally, $\sS \cap \Sigma^{-1}\widetilde{\sU} = \sE$. This is because
$\sS \cap \Sigma^{-1}\widetilde{\sU} \supseteq \sS \cap \Sigma^{-1}\sU
= \sE$ and $\sS \cap \Sigma^{-1}\widetilde{\sU} \subseteq \Ker
F(\sU)=\sE$.
\end{proof}

\begin{Theorem}
\label{silting and support tau tilting}
Assume that each object of $\sS*\Sigma\sS$ can be written as the
direct sum of indecomposable objects which are unique up to
isomorphism.  Then the maps $\Phi$ and $\Psi$ defined in Theorem~\ref{thm:main-general} are bijective.
\end{Theorem}

\begin{proof}
It suffices to show the injectivity of $\Phi$.

By Remark \ref{rmk:blanket-general}(iii), when we apply the functor
$F : \sS * \Sigma \sS \rightarrow \mod\, \sS$, we are in effect
forgetting the indecomposable direct summands which are in $\Sigma
\sS$.  So if $F( u ) \cong F( u' )$ for $u, u' \in \sS * \Sigma\sS$,
then there is an isomorphism $u \oplus \Sigma s \cong u' \oplus \Sigma
s'$ for some $s, s' \in \sS$. By the assumption in the theorem, if we
assume that $u$ and $u'$ do not have direct summands in $\Sigma\sS$,
then $u \cong u'$.

Now let $\sU$ and $\sU'$ be two presilting subcategories of $\sT$
contained in $\sS*\Sigma\sS$ such that $\Phi(\sU)=\Phi(\sU')$. Let
$\sU_1$ and $\sU'_1$ be respectively the full subcategories of $\sU$
and $\sU'$ consisting of objects without direct summands in
$\Sigma\sS$. Then $\sU=\sU_1\oplus (\sU\cap\Sigma\sS)$ and
$\sU'=\sU'_1\oplus (\sU'\cap\Sigma\sS)$. Since $\Phi(\sU)=\Phi(\sU')$,
it follows that $F(\sU_1)=F(\sU'_1)$ and
$\sU\cap\Sigma\sS=\sU'\cap\Sigma\sS$. The first equality, by the above
argument, implies that $\sU_1=\sU'_1$. Therefore $\sU=\sU'$, which
shows the injectivity of $\Phi$.
\end{proof}

\section{The Hom-finite Krull--Schmidt silting object case}

In this section, $\Bk$ is a commutative ring, $\sT$ is a triangulated
category which is essentially small, Krull--Schmidt, $\Bk$-linear and Hom-finite,
and $s \in \sT$ is a basic silting object.

We write $E = \sT( s,s )$ for the endomorphism ring and $\sS = \add( s
)$ for the associated silting subcategory.

\begin{Remark}
\label{rmk:blanket}
\begin{itemize}
  \item[(i)]  We write $\Mod\, E$ for the abelian category of
  right $E$-modules, $\mod\, E$ for the full subcategory of finitely
  presented modules, and $\prj\, E$ for the full subcategory of
  finitely generated projective modules.

\smallskip
  \item[(ii)]  Since $s$ is an additive generator of $\sS$, there is an
    equivalence 
\[
 G\; :\; \Mod\, \sS \stackrel{\sim}{\rightarrow} \Mod\, E
  \;\; , \;\;
  M \mapsto M(s)
\]
which restricts to an equivalence
\[
  \mod\, \sS \stackrel{\sim}{\rightarrow} \mod\, E
  \;\; , \;\;
  M \mapsto M(s).
\]
This permits us to move freely between the ``$E$-picture'' and the
``$\sS$-picture'' which was used in the previous section.

\smallskip

  \item[(iii)] The restricted Yoneda functor $F$ from the
    $\sS$-picture corresponds to the functor
\[
  \sT \rightarrow \Mod\, E
  \;\; , \;\;
  t \mapsto \sT( s,t )
\]
in the $E$-picture.

\smallskip



\smallskip

  \item[(iv)] By \cite[prop.\ 2.2(e)]{A} the functor $t \mapsto \sT(
    s,t )$ from (iii) restricts to an equivalence
\[
Y\; : \;  \sS \stackrel{\sim}{\rightarrow} \prj\, E.
\]

\smallskip
\noindent
Since $\sS = \add( s )$ is closed under direct sums and summands, it
is Krull--Schmidt, and it follows that so is $\prj\, E$.


\smallskip

  \item[(v)] By Remark \ref{rmk:blanket-general}(iii) the functor $t
    \mapsto \sT( s,t )$ from (iii) induces an equivalence
\begin{equation}
\label{equ:mod_equivalence2}
  ( \sS * \Sigma \sS ) / [ \Sigma \sS ]
  \stackrel{\sim}{\rightarrow} \mod\, E.
\end{equation}

\smallskip
\noindent
Since $\sS * \Sigma \sS$ is obviously closed under direct sums, and
under direct summands by Lemma~\ref{l:2-co-heart}, it is
Krull--Schmidt.  Hence so is $( \sS * \Sigma \sS ) / [ \Sigma \sS ]$
and it follows that so is $\mod\, E$.


\smallskip

  \item[(vi)] The additive category $\prj\, E$ is Krull--Schmidt by part
  (iv) and has additive generator $E_E$.  The same is hence true for
  $(\prj\, E)/[\add\, eE]$ for each idempotent $e \in E$.  It is not
  hard to check that the endomorphism ring of $E_E$ in $(\prj\,
  E)/[\add\, eE]$ is $E / EeE$, so there is an equivalence of categories
\[
  (\prj\, E)/[\add\, eE] \stackrel{ \sim }{ \rightarrow }
  \prj\, (E / EeE).
\]
In particular, $\prj\, (E / EeE)$ is Krull--Schmidt.
\end{itemize}
\end{Remark}

The following result is essentially already in \cite[prop.\ 2.16]{Ai},
\cite[start of sec.\ 5]{DF}, and \cite[prop.\ 6.1]{Wei}, all of which
give triangulated versions of Bongartz's classic proof.

\begin{Lemma}
[Bongartz Completion]
\label{lem:Bongartz}
Let $u \in \sS * \Sigma \sS$ be a presilting object.  Then there
exists an object $u^{ \prime } \in \sS * \Sigma \sS$ such that $u
\oplus u^{ \prime }$ is a silting object.
\end{Lemma}

\begin{proof}
This has essentially the same proof as classic Bongartz completion:
Since $\sT$ is $\Hom$-finite over the commutative ring $\Bk$, there is
a right $\add( u )$-approximation $u_0 \rightarrow \Sigma s$.
It gives a distinguished triangle $s \rightarrow u^{ \prime }
\rightarrow u_0 \rightarrow\Sigma s$, and it is straightforward to check that $u^{ \prime
}$ has the desired properties.
\end{proof}

The following result is essentially already in \cite[thm.\ 5.4]{DF}.

\begin{Proposition}
\label{pro:silting_criterion}
Let $u \in \sS * \Sigma \sS$ be a basic presilting object.  Then 
\[
  \mbox{ $u$ is a silting object }
  \Leftrightarrow \;
  \#_{ \sT }( u ) = \#_{ \sT }( s ).
\]
\end{Proposition}

\begin{proof}
The implication $\Rightarrow$ is immediate from \cite[thm.\
2.27]{AI} and $\Leftarrow$ is a straightforward consequence of
\cite[thm.\ 2.27]{AI} and Lemma \ref{lem:Bongartz}.
\end{proof}

As a consequence, we have

\begin{Corollary}
\label{cor:presilting-cat-has-generator}
Let $\sU$ be a  
presilting subcategory of $\sT$ contained in $\sS*\Sigma\sS$. Then there exists $u\in\sU$ such that $\sU=\add(u)$.
\end{Corollary}

\begin{proof}
Suppose on the contrary that $\sU\neq \add(u)$ for each $u \in
\sU$. Then $\sU$ contains infinitely many isomorphism classes of
indecomposable objects. In particular, there is a basic presilting
object $u \in \sU$ such that $\#_\sT (u) = \#_\sT (s)+1$. 
By Lemma \ref{lem:Bongartz}, there is an object $u'\in\sT$ such that $u\oplus u'$ is a basic silting object of
$\sT$.  Therefore, $\#_\sT(s)+1=\#_\sT(u)\leq \#_\sT(u\oplus u')=\#_\sT(s)$, a contradiction. Here the last equality follows from Proposition
\ref{pro:silting_criterion}.
\end{proof}

Theorem \ref{thm:main-general} in the current
setting combined with Corollary~\ref{cor:presilting-cat-has-generator}
immediately yields the following result.  For an object $u$ of $\sS*\Sigma\sS$, let $\Sigma u_1$ be
its maximal direct summand in $\Sigma\sS$.

\begin{Theorem}
\label{theorem:object-case-pairs}
The assignment
\[
  u \mapsto \big( \add(F(u)) , \add(u_1) \big)
\]
defines a bijection from the first to the second of the following
sets.
\begin{itemize}

  \item[(i)]  Basic presilting objects of $\sT$ which are in
    $\sS*\Sigma\sS$, modulo isomorphism.

\smallskip

  \item[(ii)] $\tau$-rigid pairs of $\mod\, \sS$.

\end{itemize}
It restricts to a bijection from the first to the second of the following sets.
\begin{itemize}

  \item[(iii)]  Basic silting objects of $\sT$ which are in
    $\sS*\Sigma\sS$, modulo isomorphism.

\smallskip

\item[(iv)]  Support $\tau$-tilting pairs of $\mod\,
  \sS$. 
\end{itemize}
\end{Theorem}

As a consequence, if $(\sM,\sE)$ is a $\tau$-rigid pair of
$\mod\,\sS$, then there is an $\sS$-module $M$ such that
$\sM=\add(M)$. 

Next we move to the $E$-picture.
Recall from Remark~\ref{rmk:blanket}(ii) and (iv) that there are equivalences
$G:\Mod\,\sS\stackrel{\sim}{\rightarrow}\Mod\, E$ and
$Y:\sS\stackrel{\sim}{\rightarrow}\prj\, E$.

\begin{Theorem}
\label{thm:main}
An $E$-module $U$ is a support $\tau$-tilting module if and
only if the pair
\[
  \big( G^{-1}(\add(U)) , Y^{-1}(\add(eE)) \big)
\]
is a support $\tau$-tilting pair of $\mod\, \sS$ for some idempotent
$e \in E$.

Consequently, the functor $\sT( s,- ) : \sT \rightarrow \Mod\, E$
induces a bijection from the first to the second of the following
sets.
\begin{enumerate}

  \item  Basic silting objects of $\sT$ which are in $\sS *
  \Sigma \sS$, modulo isomorphism.

\smallskip

  \item  Basic support $\tau$-tilting modules of $\mod\, E$, modulo
  isomorphism.

\end{enumerate}
\end{Theorem}

\begin{proof}
We only prove the first assertion.
The proof is divided into three parts. Let $u_p \in \sS*\Sigma \sS$ be
such that $u_p$ has no direct summand in $\Sigma\sS$ and
$F( u_p ) = G^{-1}( U )$.

  (a) It is clear that $U$ is a $\tau$-rigid $E$-module if and only if
  $G^{-1}(\add(U))$ is a $\tau$-rigid subcategory of $\mod\,\sS$.
  
(b) Let $e$ be an idempotent of $E$ and let $u_1\in\sS$ be such that $Y(u_1)=eE$.
We have
\[
  \begin{array}{rclcl}
 Ue & \cong & \Hom_E( eE , U ) 
      & & \\[2mm]
      & = & \Hom_{ \Mod\, \sS } \big( \sS( -,u_1 ) , F( u_p ) \big)
      & &  \\[2mm]
      & \cong & F ( u_p ) ( u_1 )
      & & \mbox{ Remark \ref{rmk:blanket-general}(ii) } 
  \end{array}
\]
Therefore $Ue = 0$ if and only if $M( u' ) = 0$ for each $M \in
\add(F(u_p)) = G^{-1}(\add(U))$ and each $u' \in \add(u_1) =
Y^{-1}(\add(eE))$. 

(c) Suppose that $(G^{-1}(\add(U)),Y^{-1}(\add(eE)))$ is a
$\tau$-rigid pair. Let $u$ be the corresponding basic presilting
object of $\sT$ as in Theorem \ref{theorem:object-case-pairs}.
More
precisely, $u=u_p\oplus \Sigma u_1$, where $u_p$ and $u_1$ are as
above. Then
\[
  \begin{array}{llcl}
    \lefteqn{\big( G^{-1}(\add(U)) , Y^{-1}(\add(eE)) \big) \mbox{ is a support $\tau$-tilting pair}} &&\hspace{10ex}& \\[2mm]
    \Leftrightarrow & \mbox{ $u$ is a silting object}
      & & \mbox{Theorem \ref{theorem:object-case-pairs}} \\[2mm]
    \Leftrightarrow & \#_{\sT}(u)=\#_{\sT}(s)
      & & \mbox{Proposition \ref{pro:silting_criterion}} \\[2mm]
    \Leftrightarrow & \#_{ \sS * \Sigma \sS }( u ) = \#_{\sS}(s) &&\\[2mm]
    \Leftrightarrow & \#_{ \sS * \Sigma \sS }( u ) = \#_{ \prj\, E}( E )
      & & \mbox{Remark \ref{rmk:blanket}(iv)}\\[2mm]
    \Leftrightarrow & \#_{ (\sS * \Sigma \sS)/[\Sigma \sS] }( u ) + \#_{ \sS * \Sigma \sS }( \Sigma u_1 ) = \#_{\prj\, E}( E ) &&
     \\[2mm]
    \Leftrightarrow & \#_{ \mod\, E }( U ) + \#_{ \prj\, E }( eE ) = \#_{\prj\, E} (E)
      & & \mbox{Remark \ref{rmk:blanket}(iv+v)}\\[2mm]
    \Leftrightarrow & \#_{ \mod\, E }( U ) =  \#_{\prj\, E} (E) - \#_{ \prj\, E }( eE ) &&\\[2mm]
    \Leftrightarrow & \#_{ \mod\, E }( U ) =  \#_{(\prj\, E)/[\add\, eE]} (E)&&\\[2mm]
    \Leftrightarrow & \#_{ \mod\, E }( U ) =  \#_{\prj\, (E/EeE)}
     (E/EeE)&& \mbox{Remark \ref{rmk:blanket}(vi)}\\[2mm]
    \Leftrightarrow & U \mbox{ is a support $\tau$-tilting module. } 
  \end{array}
\]

\end{proof}

\section{Support $\tau$-tilting pairs and torsion classes}

In this section $\Bk$ is a commutative noetherian local ring
and $\sC$ is an essentially small, Krull-Schmidt $\Bk$-linear and
Hom-finite category.

The main result in this section is the following.

\begin{Theorem}
\label{tau tilting and torsion class 2}
There is a bijection $\sM\mapsto\Fac\,\sM$ from the first to the
second of the following sets.
\begin{itemize}

  \item[(i)] Support $\tau$-tilting pairs $(\sM,\sE)$ of $\mod\,\sC$. 

\smallskip

  \item[(ii)] Finitely generated torsion classes $\sT$ of $\Mod\,\sC$
    such that each finitely generated projective $\sC$-module has a
    left $\sP(\sT)$-approximation. 

\end{itemize}
\end{Theorem}

We start with the following observation.

\begin{Lemma}
\label{tau-rigidity}
Let $\sM$ be a subcategory of $\mod\,\sC$. The following conditions
are equivalent.
\begin{itemize}

  \item[(i)] $\sM$ is $\tau$-rigid.

\smallskip

  \item[(ii)] $\Ext^1_{ \Mod\,\sC }( \sM , \Fac\,\sM ) = 0$.

\smallskip

\item[(iii)]  Each $m \in \sM$ has a minimal projective presentation
\[
  0 
  \to \Omega^2m
  \xrightarrow{d_2} P_1
  \xrightarrow{d_1} P_0
  \to m
  \to0
\]
such that for each $m' \in \sM$ and each morphism $f : P_1 \to m'$, 
there exist morphisms $a : P_0 \to m'$ and  $b : P_1 \to \Omega^2 m$ 
such that $f = ad_1 + fd_2b$.
\[
  \xymatrix{
    0 \ar[r] & \Omega^2 m \ar[r]^{d_2}
      & P_1 \ar[r]^{d_1} \ar[d]^f\ar@<1ex>[l]^b
      & P_0 \ar[r] \ar[dl]^a & m \ar[r] & 0\\
    & & m'
           }
\]
\end{itemize}
\end{Lemma}

\begin{proof}
(i)$\Rightarrow$(ii): For each $m\in\sM$, there exists a projective
presentation $P_1\xrightarrow{\pi}P_0\to m\to0$ such that
$\Hom_{\Mod\,\sC}(\pi,m')$ is surjective for each $m'\in\sM$.  Let $n
\in \Fac\, \sM$ be given and pick an epimorphism $p:m'\to n$ with
$m'\in\sM$.  To show $\Ext^1_{\Mod\,\sC}(m,n)=0$, it is enough to show
that each $f\in\Hom_{\Mod\,\sC}(P_1,n)$ factors through $\pi$.  Since
$p$ is an epimorphism and $P_1$ is projective, there exists $g:P_1\to
m'$ such that $f=pg$. Then there exists $h:P_0\to m'$ such that
$g=h\pi$ by the property of $\pi$.
\[
  \xymatrix{
    P_1\ar[r]^{\pi}\ar[dr]|(.35)f\ar[d]_g
      &P_0\ar[r]\ar[dl]|(.35)h&m\ar[r]&0\\
    m'\ar[r]_p&n
           }
\]
Thus $f=ph\pi$ holds, and we have the assertion.

\smallskip

(ii)$\Rightarrow$(iii): For each $m \in \sM$, take a minimal
projective presentation $0\to\Omega^2m\xrightarrow{d_2} P_1\xrightarrow{d_1}P_0\to
m\to0$.  Let $m' \in \sM$ and $f : P_1 \to m'$ be given, set
$n:=\Image(fd_2)$ and let $0 \to n \xrightarrow{\iota} m'
\xrightarrow{\pi} n' \to0$ be an exact sequence.  Then $\pi f:P_1\to
n'$ factors through $P_1\to\Image d_1$.  Since $n'\in\Fac\,\sM$ and
$\Ext^1_{\Mod\,\sC}(m,\Fac\,\sM)=0$, there exists $g:P_0\to n'$ such
that $gd_1=\pi f$.
\[
  \xymatrix{
    0\ar[r]&\Omega^2m\ar[r]^{d_2}\ar[d]_{f'}&P_1\ar[r]\ar[d]^f
      &\Image d_1\ar[r]\ar[d]&0\\
    0\ar[r]&n\ar[r]_{\iota}&m'\ar[r]_{\pi}&n'\ar[r]&0
           }\ \ \ \ \
  \xymatrix{
    0\ar[r]&\Image d_1\ar[r]\ar[d]&P_0\ar[r]\ar[dl]^g&m\ar[r]&0\\
    &n'
           }
\]
Since $\pi$ is an epimorphism and $P_0$ is projective, there exists
$a:P_0\to m'$ such that $g=\pi a$.  Since $\pi(f-ad_1)=0$, there
exists $h:P_1\to n$ such that $f=ad_1+\iota h$.  Since $f'$ is
surjective (by definition of $n$) and $P_1$ is projective,
there exists $b:P_1\to\Omega^2m$ such that $h=f'b$.
\[
  \xymatrix{
    0\ar[r]&\Omega^2m\ar[r]^{d_2}\ar[d]^{f'}
      &P_1\ar[r]^{d_1}\ar[d]^f\ar@<1ex>[l]^b\ar[dl]^h
      &P_0\ar[r]\ar[dl]^a\ar[d]^g&m\ar[r]&0\\
    0\ar[r]&n\ar[r]_{\iota}&m'\ar[r]_{\pi}&n'\ar[r]&0
           }
\]
Then we have $f = ad_1 + \iota f'b = ad_1 + fd_2b$.
\smallskip

(iii)$\Rightarrow$(i): For each $m \in \sM$, take a minimal
projective presentation
$0\to\Omega^2m\xrightarrow{d_2} P_1\xrightarrow{d_1}P_0\to m\to0$
satisfying the assumption in (iii).
We need to show that each $f:P_1\to m'$ with
$m'\in\sM$ factors through $d_1$.  By our assumption, there exist
$a:P_0\to m'$ and $b:P_1\to\Omega^2m$ such that $f=ad_1+fd_2b$.
Applying our assumption to $fd_2b:P_1\to m'$, there exist $a':P_0\to
m'$ and $b':P_1\to\Omega^2m$ such that $fd_2b=a'd_1+fd_2bd_2b'$.  Thus
$f=(a+a')d_1+fd_2bd_2b'$ holds.  Repeating a similar argument gives
\[
  \Hom_{\Mod\,\sC}(P_1,M)
  = \Hom_{\Mod\,\sC}(P_0,M)d_1
    + \Hom_{\Mod\,\sC}(P_1,M)(\rad\End_{\Mod\,\sC}(P_1))^n
\]
for each $n \geq 1$ since $d_2 \in \rad \Hom_{\Mod\, \sC}( \Omega^2 M
, P_1 )$.  Since $\sC$ is Hom-finite over $\Bk$, we have
$(\rad\End_{\Mod\,\sC}(P_1))^\ell\subset\End_{\Mod\,\sC}(P_1)(\rad\Bk)$
for sufficiently large $\ell$. Thus we have
\[
  \Hom_{\Mod\,\sC}( P_1 , M )
  = \bigcap_{n\ge0} \big( \Hom_{\Mod\,\sC}( P_0 , M )d_1
                     + \Hom_{\Mod\,\sC}( P_1 , M )( \rad\Bk )^n \big).
\] 
The right hand side is equal to $\Hom_{\Mod\,\sC}(P_0,M)d_1$ itself by Krull's intersection Theorem \cite{M}.
\end{proof}

\begin{Proposition}
\label{recover}
Let $(\sM,\sE)$ be a support $\tau$-tilting pair of $\mod\,\sC$.  Then
$\Fac\,\sM$ is a finitely generated torsion class with
$\sP(\Fac\,\sM)=\sM$.
\end{Proposition}

\begin{proof}
(i) We show that $\Fac\,\sM$ is a torsion class. 
Clearly $\Fac\,\sM$ is closed under factor modules. 
We show that $\Fac\,\sM$ is closed under extensions.
Let $0\to x\to y\xrightarrow{f} z\to0$ be an exact sequence in $\Mod\,\sC$ such that $x,z\in\Fac\,\sM$.
Take an epimorphism $p:m\to z$ with $m\in\sM$. Since $\Ext^1_{\Mod\,\sC}(m,x)=0$ holds by Lemma \ref{tau-rigidity}(ii), we have that $p$ factors through $f$.
Thus we have an epimorphism $x\oplus m\to y$, and $y\in\Fac\,\sM$ holds.
Hence $\Fac\,\sM$ is a torsion class.

\smallskip

(ii) Since $\Ext^1_{\Mod\,\sC}(\sM,\Fac\,\sM)=0$ holds by Lemma
\ref{tau-rigidity}(ii), each object in $\sM$ is Ext-projective in
$\Fac\,\sM$.  It remains to show that if $n$ is an Ext-projective
object in $\Fac\,\sM$, then $n\in\sM$.  Let
$P_1\xrightarrow{f}P_0\xrightarrow{e} n\to0$ be a projective
presentation.  Since $\sM$ is support $\tau$-tilting, there exist
exact sequence $P_i\xrightarrow{g_i}m_i\xrightarrow{h_i}m'_i\to0$ with
$m_i,m'_i\in\sM$ and a left $\sM$-approximation $g_i$ for $i=0,1$.

Let $\overline{\sC}:=\sC/\ann\sM$ for the annihilator ideal $\ann\sM$ of $\sM$ and $\overline{P_i}:=P_i\otimes_{\sC}\overline{\sC}$.
Then we have induced exact sequences $0\to\overline{P_i}\xrightarrow{g_i}m_i\xrightarrow{h_i}m'_i\to0$ for $i=0,1$ and
$\overline{P_1}\xrightarrow{f}\overline{P_0}\xrightarrow{e} n\to0$.
We have a commutative diagram
\[
  \xymatrix{
    0\ar[r]&\overline{P_1}\ar[r]^{g_1}\ar[d]^f&m_1\ar[r]^{h_1}\ar[d]^a&m'_1\ar[r]\ar[d]^{b}&0\\
    0\ar[r]&\overline{P_0}\ar[r]_{g_0}&m_0\ar[r]_{h_0}&m'_0\ar[r]&0
           }
\]
of exact sequences. Taking a mapping cone, we have an exact sequence
\[
  \xymatrix{
    0\ar[r]
      & \overline{P_1}\ar[rr]^-{{\left[\begin{smallmatrix}g_1\\
              f\end{smallmatrix}\right]}}
      & & m_1\oplus\overline{P_0}\ar[rr]^-{
{\left[\begin{smallmatrix}h_1 & 0 \\ a & -g_0\end{smallmatrix}\right]}
                                        }
      & & m'_1\oplus m_0\ar[rr]^-{{\left[\begin{smallmatrix}b
      & -h_0\end{smallmatrix}\right]}}
      & & m'_0\ar[r]
      & 0.
           }
\]
Since $\Ext^1_{\Mod\,\sC}(m'_0,n)=0$ holds by Lemma \ref{tau-rigidity}(ii), we have the following commutative diagram.
\[
  \xymatrix{
    &0\ar[r]&\overline{P_1}\ar[r]^-{{\left[\begin{smallmatrix}g_1\\ f\end{smallmatrix}\right]}}\ar@{=}[d]&m_1\oplus\overline{P_0}\ar[rr]^-{{\left[\begin{smallmatrix}h_1&0\\ a&-g_0\end{smallmatrix}\right]}}\ar[d]^{{\left[\begin{smallmatrix}0&1\end{smallmatrix}\right]}}&&m'_1\oplus m_0\ar[rr]^-{{\left[\begin{smallmatrix}b&&-h_0\end{smallmatrix}\right]}}\ar@{.>}[d]&&m'_0\ar[r]&0\\
    0\ar[r]&\Ker f\ar[r]&\overline{P_1}\ar[r]_f&\overline{P_0}\ar[rr]_e&&n\ar[rr]&&0
            }
\]
Taking a mapping cone, we have an exact sequence
\[\xymatrix{
0\ar[r]&\overline{P_1}\oplus\Ker f\ar[r]&m_1\oplus\overline{P_0}\oplus\overline{P_1}\ar[r]&m'_1\oplus m_0\oplus\overline{P_0}\ar[r]&m'_0\oplus n\ar[r]&0.
}\]
Cancelling a direct summand of the form
$\overline{P_1}\xrightarrow{{\left[\begin{smallmatrix}0\\ 1\end{smallmatrix}\right]}}
\overline{P_0}\oplus \overline{P_1}\xrightarrow{{\left[\begin{smallmatrix}1&0\end{smallmatrix}\right]}}\overline{P_0}$, we have an exact sequence
\[\xymatrix{
0\ar[r]&\Ker f\ar[r]&m_1\ar[r]^-c&m'_1\oplus m_0\ar[r]^-d&m'_0\oplus n\ar[r]&0.
}\]
Since $\Image c\in\Fac\,\sM$ and $m'_0\oplus n$ is Ext-projective in $\Fac\,\sM$, the epimorphism $d$ splits.
Thus $n \in \sM$ as desired.
\end{proof}

Now we are ready to prove Theorem \ref{tau tilting and torsion class 2}.

Let $\sM$ be a support $\tau$-tilting subcategory of $\mod\, \sC$. By definition, each representable $\sC$-module has a left $\sM$-approximation. Since $\sP(\Fac\,\sM)=\sM$ holds by
Proposition \ref{recover}, the map $\sM\mapsto\Fac\,\sM$ is
well-defined from the set (i) to the set (ii) and 
it is injective.

We show that the map is surjective.  For $\sT$ in the set described in
(ii), let $\sE := \bigcap_{ m \in \sT} \Ker m$ and $\sM := \sP( \sT
)$.  We will show that $( \sM , \sE )$ is a support $\tau$-tilting pair of $\mod\, \sC$.
Since $\Ext^1_{ \Mod\, \sC }( \sM , \sT ) = 0$ and $\Fac\, \sM
\subset \sT$ hold, it follows from Lemma \ref{tau-rigidity} that
$\sM$ is $\tau$-rigid.  For $s \in \sC$, take a left
$\sM$-approximation $\sC(-,s)\xrightarrow{f} m$.

It remains to show $\Coker f\in\sM$. Since
$\sP(\Fac\, \sM)=\sM$ holds by Proposition \ref{recover},
we only have to show $\Ext^1_{\Mod\,\sC}(\Coker f,m')=0$ for each $m'\in\sM$.
Let $f=\iota\pi$ for $\pi:\sC(-,s)\to\Image f$ and $\iota:\Image f\to
m$. Applying $\Hom_{\Mod\,\sC}(-,m')$ to the exact sequence
$0\to\Image f\xrightarrow{\iota}m\to \Coker f\to0$, we have an exact
sequence
\begin{eqnarray*}
  \Hom_{\Mod\,\sC}(m,m')\xrightarrow{\iota^*}\Hom_{\Mod\,\sC}(\Image f,m')
  &\to&\Ext^1_{\Mod\,\sC}(\Coker f,m')\\
  &\to&\Ext^1_{\Mod\,\sC}(m,m')=0.
\end{eqnarray*}
Let $g:\Image f\to m'$ be a morphism in $\Mod\,\sC$.
Since $f$ is a left $\sM$-approximation, there exists $h:m\to m'$ such that $g\pi=hf$.
Then $g=h\iota$ holds. Thus $\iota^{\ast}
:\Hom_{\Mod\,\sC}(m,m')\to\Hom_{\Mod\,\sC}(\Image f,m')$ is
surjective, and we have $\Ext^1_{\Mod\,\sC}(\Coker f,m')=0$.
Consequently we have $\Coker f\in\sP(\sT)=\sM$. Thus the assertion follows.
\qed




\begin{thebibliography}{19}

\bibitem{AIR}  T.\ Adachi, O.\ Iyama, and I.\ Reiten, {\it
    $\tau$-tilting theory}, Compositio Math.\ {\bf 150} (2014),
  415-–452. 

\bibitem{Ai}  T.\ Aihara, {\it Tilting-connected symmetric
  algebras}, Algebra.\ Represent.\ Theor. {\bf 16} (2013), no.~3, 873--894.

\bibitem{AI} T.\ Aihara and O.\ Iyama, {\it Silting mutation in
    triangulated categories}, J.\ London Math.\ Soc. (2) {\bf 85}
  (2012), 633--668.

\bibitem{AF}  F.\ W.\ Anderson and K.\ R.\ Fuller, ``Rings and
categories of modules'', Grad.\ Texts in Math., Vol.\ 13,
Springer, New York, 1974.

\bibitem{A}  M.\ Auslander, {\it Representation theory of Artin
algebras I}, Comm.\ Algebra {\bf 1} (1974), 177--268.

\bibitem{B}  H.\ Bass, ``Algebraic K-theory'', W.\ A.\ Benjamin Inc.,
New York--Amsterdam, 1968.

\bibitem{BR}  A.\ Beligiannis and I.\ Reiten, 
``Homological and homotopical aspects of torsion theories'', Mem.\
Amer.\ Math.\ Soc.\ {\bf 188} (2007), no.\ 883.

\bibitem{Bondarko10}
M. V. Bondarko, \emph{Weight structures vs. {$t$}-structures; weight filtrations,
  spectral sequences, and complexes (for motives and in general)}, J. K-Theory
  \textbf{6} (2010), no.~3, 387--504.


\bibitem{DF}  H.\ Derksen and J.\ Fei, {\it General presentations of
  algebras}, preprint (2009).  {\tt math.RA/0911.4913v2. }

\bibitem{HRS}  D.\ Happel, I.\ Reiten, and S. Smal\o,
``Tilting in abelian categories and quasitilted algebras'',
Mem.\ Amer.\ Math.\ Soc.\ {\bf 120} (1996), no.\ 575.


\bibitem{IY}  O.\ Iyama and Y.\ Yoshino, {\it Mutation in
triangulated categories and rigid Cohen-Macaulay mo\-du\-les},
Invent.\ Math.\ {\bf 172} (2008), 117--168.

\bibitem{K} B.\ Keller, {\it The periodicity conjecture for pairs of
    Dynkin diagrams},  Ann.\ of Math.\ (2) {\bf 177}(2013), no.~1, 111--170.

\bibitem{M} H.\ Matsumura, ``Commutative ring theory'',
Cambridge Studies in Advanced Mathematics, Vol. 8.
Cambridge University Press, Cambridge, 1989.

\bibitem{MSSS}  O.\ Mendoza, E.\ C.\ Saenz, V.\ Santiago, and M.\ J.\
Souto Salorio, {\it Auslander-Buchweitz context and co-t-structures},
Appl.\ Categ.\ Structures\ {\bf 21} (2013), 417--440.

\bibitem{P}  D.\ Pauksztello, {\it Compact corigid objects in
triangulated categories and co-t-structures}, Cent.\ Eur.\ J.\ Math.\
{\bf 6} (2008), 25--42.

\bibitem{Wei}  J.\ Wei, {\it Semi-tilting complexes},
Israel J.\ Math.\ {\bf 194} (2013), 871--893. 

\bibitem{W}  J.\ Woolf, {\it Stability conditions, torsion theories
and tilting}, J.\ London Math.\ Soc.\ {\bf 82} (2010), 663--682.

\end{thebibliography}
\end{document}